\RequirePackage{fix-cm}
\documentclass{svjour3}                     
\smartqed  
\usepackage{graphicx}
\usepackage{amsfonts}
\usepackage{amsmath}
\usepackage{amssymb}
\usepackage{enumerate}
\usepackage{cite}
\newcommand{\largesquare}{\scalebox{1.4}{$\Box$}}

\newcommand{\repcycle}[1]{\largesquare$_{\alpha = 1}^{#1} C_{2m_{\alpha}}$}
\newcommand{\reppath}[1]{\largesquare$_{\alpha = 1}^{#1} P_{2m_{\alpha}}$}
\newcommand{\repcube}[2]{\largesquare$_{\alpha = 1}^{#1} K_{#2, #2}$}

\begin{document}

\title{The minimum orientable genus of the repeated Cartesian product of graphs
}


\author{Marietta Galea         \and
        John Baptist Gauci 
}


\institute{M. Galea \at
              Department of Mathematics, University of Malta, Msida, Malta\\
              \email{marietta.galea@gmail.com}           
           \and
           J.B. Gauci \at
              Department of Mathematics, University of Malta, Msida, Malta\\
              \email{john-baptist.gauci@um.edu.mt}
}

\date{Received: date / Accepted: date}

\maketitle

\begin{abstract}
Determining the minimum genus of a graph is a fundamental optimisation problem in the study of network design and implementation as it gives a measure of non-planarity of graphs. In this paper, we are concerned with determining the smallest value of $g$ such that a given graph $G$ has an embedding on the orientable surface of genus $g$. In particular, we consider the Cartesian product of graphs since this is a well studied graph operation which is often used for modeling interconnection networks. The $s$-cube $Q_i^{(s)}$ is obtained by taking the repeated Cartesian product of $i$ complete bipartite graphs $K_{s,s}$. We determine the genus of the Cartesian product of the $2r$-cube with the repeated Cartesian product of cycles and  of the Cartesian product of the $2r$-cube with the repeated Cartesian product of paths.
\end{abstract}

\section{Introduction}

Ever since their introduction by Sabidussi in 1959 \cite{Sabidussi1959}, Cartesian product graphs have been applied in many areas, including in the study of interconnection networks and coding theory. An interconnection network
facilitates the transportation of data in a parallel system of processors. Such a system is used to break up a complex task into parallel tasks so that these tasks can be performed concurrently by the separate processors, hence improving system
performance (see, for example, \cite{Wu1981} for a more detailed discussion). Coding theory studies properties of codes, including the development of error-detecting and error-correcting codes and effective data transmission and data storage methods. Many important classes of graphs such as hypercubes, Hamming graphs, and prisms, are Cartesian products. This graph product provides a quick, effective and efficient way of constructing bigger graphs from smaller ones in such a way that the structure of the smaller graphs is preserved. Hence, Cartesian product graphs play a key role in the design and analysis of interconnection networks.

Various disciplines which deal with network design are interested in determining the optimal surface on which a network can be drawn without any of the links between its nodes intersecting each other. Mathematically, this problem revolves around determining the surface of minimum genus that permits a given graph to be drawn on it without edge-crossings. In computer science, the minimum genus of a graph is also a useful parameter in algorithm design since faster algorithms can be designed when the topological structure of a graph is known. These algorithms usually assume that the input graph is embedded on some surface, as for example in \cite{EricksonEtAl2012}. The problem of determining the minimum genus of a graph, albeit very relevant and applicable, has been proved to be substantially difficult from the theoretical, practical and structural perspectives. Garey and Johnson \cite{GareyJohnson1979} listed its complexity as one of the most important open problems.

The problem of embedding graphs on surfaces other than the sphere dates back to the Map Colour Conjecture due to Heawood in 1890 \cite{Heawood1890}. We present the following definitions related to surfaces.

\begin{definition} \cite[pp.~730--731]{HandbookOfGraphTheory}
\begin{enumerate} 
  \item The \emph{open unit disk} is the subset $\{(x,y):x^2+y^2<1\}$ and an \emph{open disk} is any topological space homeomorphic to the open unit disk.
  \item The \emph{unit half-disk} is the subset $\{(x,y):x\geq 0, x^2+y^2<1\}$ and a \emph{half-disk} is any topological space homeomorphic to the unit half-disk.
  \item The \emph{closed unit disk} is the subset $\{(x,y):x^2+y^2\leq 1\}$ and an \emph{closed disk} is any topological space homeomorphic to the closed unit disk.
  \item A \emph{2-manifold} is a topological space in which each point has a neighbourhood that is homeomorphic either to an open disk or to a half-disk.
  \item A \emph{surface} $S$ is a 2-manifold that is compact and without boundary.
\end{enumerate}
\end{definition}

Three of the standard surfaces are the sphere, the torus and the projective plane, respectively denoted by $S_0$, $S_1$ and $N_1$.

\begin{definition} \cite[pp.~733--734]{HandbookOfGraphTheory}
\begin{enumerate}
  \item Let $S$ and $S'$ be two surfaces. The \emph{connected sum $S\# S'$} is obtained by excising the interior of a closed disk in each surface and then gluing the corresponding boundary curves.
  \item A \emph{handle} is added to a surface $S$ by forming the connected sum $S \# S_1$, whereas a \emph{crosscap} is added to a surface $S$ by forming the connected sum $S \# N_1$.
  \item The \emph{orientable surface with $g$ handles}, denoted by $S_g$, is the connected sum of $g$ copies $S_1$. The \emph{genus} of $S_g$ is the number $g$ of handles.
  \item The \emph{nonorientable surface with $k$ crosscaps}, denoted by $N_k$, is the connected sum of $k$ copies of $N_1$. The \emph{nonorientable genus} of $N_k$ is the number $k$ of crosscaps.
\end{enumerate}
\end{definition}

In this work we restrict our attention to the orientable surfaces $S_g$. The reader might find it easier to visualise the adding of a handle to the sphere by removing two disjoint open discs from the sphere and identifying their boundaries with the ends of a truncated cylinder (refer to Figure \ref{FigTorus}). The solid so obtained is, in fact, the torus $S_1$. Repeating this process of attaching $g$ handles to the sphere $S_0$ yields the orientable surface $S_g$.

  \begin{figure}[h!]
        \centering
       \includegraphics[width=\textwidth]{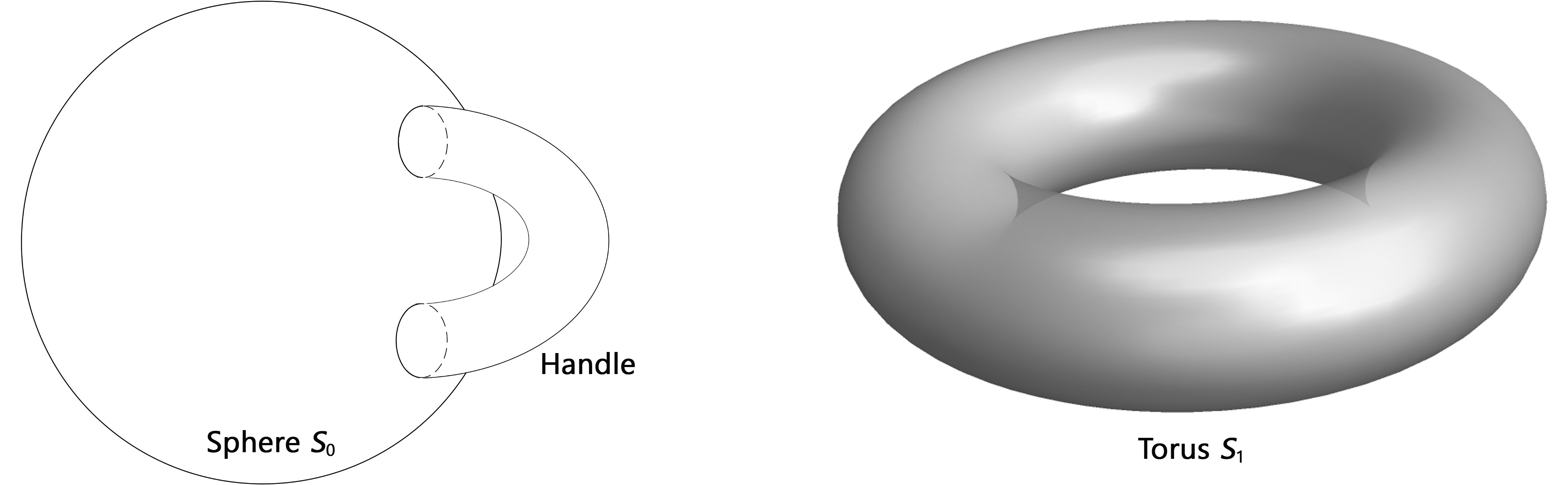}
        \caption{Adding a handle to $S_0$ results in the torus $S_1$.}
        \label{FigTorus}
    \end{figure}

An \textit{embedding} of a graph $G$ is a drawing of the graph $G$ on a surface such that no two edges cross. If every face in a drawing of $G$ is homeomorphic to an open disc, the embedding is referred to as a \textit{2-cell embedding}. Such an embedding in which each face has four sides is called a \textit{quadrilateral embedding}. Given a graph $G$, if $G$ has an embedding on the surface $S_g$ but not on the surface $S_{g-1}$, then $G$ is said to be of \textit{genus} $g$ and we write $g(G)=g$. If $G$ is connected, then an embedding of a graph $G$ of genus $g$ on the surface $S_g$ is a 2-cell embedding \cite{Youngs}. 

It is useful to note that every graph can be embedded on some orientable surface. This is because, if we start with a drawing of a graph on the sphere such that some  of the edges cross, then we can suitably add a handle at each point of intersection so that one of the crossed edges passes over the handle, hence eliminating the respective crossing. Thus, handles are used to remove crossings from a drawing of a graph on a surface.

Thomassen \cite{Thomassen1989} showed that determining the genus of a graph is NP-complete. Nevertheless, results on the genus can be obtained by considering families of graphs with specific properties. 
Two properties which we shall make frequent reference to are connectedness and regularity. A graph is \emph{connected} if there is a path between any two vertices in the graph. A graph is \emph{regular} is all the vertices of the graph have the same degree. For standard graph theoretical terminology we refer the reader to \cite{Chromatic-GT}.

Ringel \cite{Ringel} provided one of the first results in this area when he determined the genus of the complete bipartite graph $K_{m,n}$. White \cite{White1970} considered Cartesian product graphs, where the \textit{Cartesian product} $G_1\square G_2$ of two graphs $G_1$ and $G_2$ has vertex set $V(G_1\square G_2)=\{(v_1,v_2):v_1 \in V(G_1) \text{ and } v_2\in V(G_2)\}$ and edge set $E(G_1\square G_2)=\{(v_1,v_2)(w_1,w_2): (v_1=w_1 \text{ and }  v_2w_2\in E(G_2)) \text{ or } (v_1w_1\in E(G_1) \text{ and } v_2=w_2)\}$. He determined the genus of the Cartesian product of the complete bipartite graph $K_{2m,2m}$ with $K_{2n,2n}$, for all natural numbers $m$ and $n$. White \cite{White1970} also considered three families of graphs obtained by taking repeated Cartesian products, namely the repeated Cartesian product of paths, the repeated Cartesian product of cycles, and the repeated Cartesian product of complete bipartite graphs, respectively defined more precisely as follows:
\begin{gather*}
   H_j=H_{j-1} \square P_{m_j}  \text{where } j\geq 2 \text{ and } H_1=P_{m_1},\\
   G_j=G_{j-1} \square C_{2m_j}  \text{where } j\geq 2 \text{, } m_j \geq 2 \text{ for all } j, \text{ and } G_1=C_{2m_1},\\
   Q_j^{(t)}=Q_{j-1}^{(t)} \square K_{t,t}  \text{where } j\geq 2 \text{ and } Q_1^{(t)}= K_{t,t}.
\end{gather*}

We remark that
\begin{enumerate}[(i)]
  \item $P_n$ and $C_n$ denote the path and the cycle on $n$ vertices, respectively;
  \item the well-known hypercube  $Q_n$ is $H_n$ with $m_j=2$ for all $j\in\{1,2,\ldots,n\}$, whose genus is given by $1+2^{n-3}(n-4)$ (see \cite{BeinekeHarary1965, Ringel1955, White1970} and more recently \cite{HammackKainen2021});
  \item $Q_i^{(t)}$, also referred to as the $t$-cube, is a generalisation of the hypercube $Q_n$ since $Q_n^{(1)}=Q_n$.
\end{enumerate}

Instances of the Cartesian product of graphs were recently studied in \cite{MillichapSalinas2022} and \cite{Sun2023}, where the genus of grid graphs and of the $n$-prism were respectively determined.

Henceforth, to make the results more self-contained, we replace the notation used by White \cite{White1970} with the following:

\begin{gather*}
    H_j = \Box_{\alpha = 1}^j P_{m_{\alpha}},\\
    G_j = \Box_{\alpha = 1}^j C_{2m_{\alpha}},\\
    Q_j^{(t)} = \Box_{\alpha = 1}^j K_{t,t}.\\
\end{gather*}

Pisanski \cite{Pisanski1980} generalised White's result by considering the repeated Cartesian product of connected regular bipartite graphs on an even number of vertices. In particular, letting $G(n,d)$, for $n\geq d\geq 1$, denote a connected regular bipartite graph of degree $d$ on $2n$ vertices, Pisanski determined the genus of the repeated Cartesian product $G(n,d)\square G_1(n_1,d_1)\square G_2(n_2,d_2)\square \ldots \square$ $G_m(n_m,d_m)$ provided that $\max\{d_1,d_2,\ldots,d_m\}\leq d \leq d_1+d_2+\ldots+d_m$.\\

The aim of this paper is to generalise further the work done on the genus of the repeated Cartesian product of graphs. In our main results, respectively stated in Theorem \ref{thm:mainResult} and Theorem \ref{thm:mainResult2} hereunder, we determine the minimum genus of the Cartesian products $$\left(\text{\repcube{j_1}{2r}} \right) \square \left(\text{\repcycle{j_2}}\right) \textrm{~and~} \left(\text{\repcube{j_1}{2r}}\right)\square \left(\text{\reppath{j_2}}\right),$$ where  $j_1$ and $j_2$ are positive integers. We remark that most of these cases are excluded from the result by Pisanski due to the condition imposed therein on the degrees, as highlighted later in this work.

\section{Basic results}

An important and well known result about the Cartesian product of two graphs is given in the following lemma.

\begin{lemma}\label{thm:prodBip}\cite{sabidussi_1957} Let $G$ and $H$ be two non-empty graphs. The Cartesian product $G\square H$ is bipartite if and only if $G$ and $H$ are bipartite.
\end{lemma}

It follows immediately that the repeated Cartesian product of bipartite graphs is also bipartite. The following is a well-known result relating the genus of a surface and the number of vertices, edges and faces of a graph $G$ which has a $2$-cell embedding on this surface.

\begin{lemma}\label{thm:euler-gen-formula}
    \cite{Euler} (Euler's Generalised Formula) Let $G$ be a connected graph with $n$ vertices and $m$ edges. Consider a $2$-cell embedding of $G$ on an orientable surface $S_g$ with genus $g$, resulting in $f$ faces. Then, $$n + f - m = 2 - 2g.$$
\end{lemma}

A direct consequence of Euler's generalized formula for 2-cell embeddings of a bipartite graph on any surface is stated and proved below.

\begin{lemma}\label{thm:genus}
  If a bipartite graph $G$ with $n$ vertices and $m$ edges has a quadrilateral embedding on a surface, then $g(G)=1+\frac{m}{4}-\frac{n}{2}$.
\end{lemma}
\begin{proof}
     Since $G$ is bipartite, the smallest possible number of edges that can border a face is four, implying that a quadrilateral embedding of $G$ is minimal. Therefore, by double-counting, $4f = 2m$. Substituting this value in Lemma \ref{thm:euler-gen-formula}, we get $g(G)=1+\frac{m}{4}-\frac{n}{2}$.
     \qed
\end{proof}
In \cite{White1970}, White proved the following results.

\begin{proposition} \cite{White1970} ~\
Let $H_j =$ \reppath{j}, $G_j =$ \repcycle{j} and $Q_j^{(r)} =$ \repcube{j}{r}. Then,
\begin{enumerate}[1.]
\item   $g(H_j)=1+\frac{1}{4}\left(\prod_{k=1}^{j}m_k\right)\left(j-2-\sum_{k=1}^{j}\frac{1}{m_k}\right)$  for $j\geq 3$  and even $m_1,m_2,m_3$;
\item   $g(G_j)=1+2^{(j-2)}(j-2)\prod_{k=1}^{j}m_k$  for $j\geq 2$;
\item   $g\left(Q_j^{(r)}\right)=1+2^{j-3}r^j(j-4)$  for $r$ even and $j\geq 1$ or for $r\in\{1,3\}$  and $j\geq 2$.
\end{enumerate}
\end{proposition}

Ringel \cite{Ringel} proved the following result, which will be useful in the sequel.

\begin{proposition} \label{theo-Ringel} \cite{Ringel}
    The complete bipartite graph $K_{2r,2r}$ has an embedding on a surface of genus $(r-1)^2$.
\end{proposition}

Another useful result by White  in \cite{White1970}, given below, establishes a partition of a set of faces found in a quadrilateral embedding for $K_{2r,2r}$.
\begin{proposition} \cite{White1970} \label{lem:WhitePartitionK2s2s}
    For a quadrilateral embedding of $K_{2r,2r}$, the set of $2r^2$ quadrilateral faces may be partitioned into $2r$ subsets of $r$ faces each so that each subset of $r$ faces contains all $4r$ vertices of the graph.
\end{proposition}

We end this section by proving the following result which will be needed in the proofs that follow.

\begin{lemma}\label{additionalQuadFaces}
Let $F_1$ and $F_2$ be two quadrilateral faces in an embedding of a graph $G$ on a surface $S_g$ and let $C_1$ and $C_2$ be the two cycles bounding $F_1$ and $F_2$, respectively. Then adding a handle between $F_1$ and $F_2$ carrying four edges such that each edge joins a vertex of $C_1$ to the corresponding vertex of $C_2$ increases the number of quadrilateral faces of the embedding of the graph induced by $G$ on $S_{g+1}$ by two. 
\end{lemma}

\begin{proof}
Consider an embedding of $G$ on the surface $S_g$ such that $F_1$ and $F_2$ are two quadrilateral faces bounded by the cycles $C_1$ and $C_2$, respectively. When adding a handle between $F_1$ and $F_2$, we remove an open disc from each face and identify their boundaries with the ends of a truncated cylinder, thus loosing the two quadrilateral faces bounded by $C_1$ and $C_2$. Introducing four edges on the handle such that each edge is adjacent to a vertex on $C_1$ and the corresponding vertex on $C_2$ (such that no crossing between edges is introduced), we get four new faces each of which is a quadrilateral. Thus, the net increase in the number of quadrilateral faces in the embedding of the graph induced by $G$ on $S_{g+1}$ is of two. \qed
\end{proof}

\section{The Cartesian product of the $2r$-cube with cycles}
In this section we use a technique developed by Pisanski \cite{Pisanski1982} and White \cite{White1970}, referred to as the White-Pisanski method, that constructs a quadrilateral embedding for the Cartesian product of two regular bipartite graphs.


Note that the number of vertices of $Q_i^{(2r)}=$ \repcube{i}{2r} is $2^{2i}r^i$ and the number of edges is $i2^{2i}r^{i+1}$, which can be easily derived using induction or the Handshaking Lemma (namely that the sum of the degrees of the vertices of a graph $G$ is equal to twice the number of edges of $G$ \cite{Chromatic-GT}). We first construct a quadrilateral embedding of $\left( \text{\repcube{i}{2r}}\right) \Box\ C_{2s}$ and then derive the genus of this graph using Lemma \ref{thm:genus}.
\begin{lemma}\label{thm:genus2jCubeC2s}
    For $s \geq 2$, the genus of $\left( \text{\repcube{i}{2r}}\right) \Box\ C_{2s} $ is given by $$g\left( \left( \text{\repcube{i}{2r}}\right) \Box\ C_{2s}\right) = 1 + 2^{2i - 1}sr^i(ir - 1).$$
\end{lemma}
\begin{proof}
Recall that \repcube{i}{2r} is denoted by $Q_i^{(2r)}$. First, we minimally embed a copy of $Q_{i}^{(2r)}$ as follows.
    By Lemma \ref{thm:prodBip}, $K_{2r,2r}\Box \ K_{2r,2r}$ is a bipartite graph. Using Proposition \ref{theo-Ringel}, we embed $4r$ disjoint copies of $K_{2r,2r}$ on $4r$ surfaces of genus $(r-1)^2$ such that each embedding is quadrilateral. We partition these copies in two equal parts, corresponding to the vertex-set partition of $K_{2r,2r}$, such that the resulting embeddings of one partition are mirror images of the embeddings in the other partition. 
    
    All that is left is to add edges between the $4r$ disjoint copies to acquire $K_{2r,2r} \Box \ K_{2r,2r}$. Note that $4r$ edges are required to be added between each pair of oppositely oriented copies, which can be done by adding $r$ handles, each carrying four edges. Such an operation will be referred to as a link. For each surface, there are $2r$ links to be made and by Proposition \ref{lem:WhitePartitionK2s2s}, each link is designated a subset of faces from the partition. We just need to check that after partitioning the faces, the corresponding subsets can be matched correctly. For $1 \leq \overleftarrow{j} \leq 2r$, each copy $\overleftarrow{j}$ of $K_{2r,2r}$ is found in one partition, whilst for $1 \leq \overrightarrow{j} \leq 2r$ each copy $\overrightarrow{j}$ of $K_{2r,2r}$ is in the other partition. For some copy $\overleftarrow{j}$, its face partition consists of $2r$ subsets of $r$ faces each. Match subset $\overleftarrow{k}$ of the face partition to subset $\overrightarrow{k}$ in copy $\overrightarrow{j} + \overrightarrow{k}, 1 \leq \overrightarrow{j} + \overrightarrow{k} \leq 2r$ (mod $2r$). If $\overrightarrow{j} + \overrightarrow{k} = \overrightarrow{j}' + \overrightarrow{k}'$ with $\overrightarrow{k} = \overrightarrow{k}'$, then, $\overleftarrow{j} = \overleftarrow{j}'$, therefore, each subset of the partition has exactly one handle attached at each face in that subset. Each handle carries four edges, so that, by Lemma \ref{additionalQuadFaces}, the number of quadrilateral faces increases by two. This process is repeated by successively introducing the complete bipartite graph $K_{2r,2r}$ until a quadrilateral embedding of $Q_{i}^{(2r)}$ is obtained.

     Now, to embed $Q_{i}^{(2r)}\Box \ C_{2s}$, embed $2s$ copies of $Q_{i}^{(2r)}$ as described above. We partition these copies in two equal parts, corresponding to the vertex-set partition of $C_{2s}$, such that the resulting embeddings of one partition are mirror images of the embeddings in the other partition. From each copy of $Q_{i}^{(2r)}$ in one partition, two links are to be made, both to copies of the other partition. For each link, $2^{2i}r^i$ edges are to be added and these are passed over $2^{2i - 2}r^i$ handles with each handle carrying four edges. By the construction given above and Proposition \ref{lem:WhitePartitionK2s2s}, each copy of $Q_{i}^{(2r)}$ has $2r$ disjoint sets of $2^{2i - 2}r^i$ mutually vertex-disjoint quadrilateral faces that cover all $2^{2i}r^i$ vertices of $Q_{i}^{(2r)}$. Since only two links are to be made from each copy of $Q_{i}^{(2r)}$, we choose any two of the available sets. The number of quadrilateral faces in each set is precisely the amount required to perform a link, as previously described. Since the embedding of each copy of $Q_{i}^{(2r)}$ is quadrilateral and each insertion of a handle introduces more quadrilateral faces, the obtained embedding of $Q_{i}^{(2r)}\Box \ C_{2s}$ is quadrilateral. Applying Lemma \ref{thm:genus}, we get
    \begin{equation*}
        \begin{split}
            g(Q_{i}^{(2r)}\Box \ C_{2s}) &= 1 + \frac{2sir^{i+1}2^{2i} + 2sr^{i}2^{2i}}{4} - \frac{2sr^i2^{2i}}{2}\\
            &= 1 + 2^{2i - 1}sr^i(ir - 1).
        \end{split}
    \end{equation*}
\qed
\end{proof}

Let $M^{(j)} = \prod\limits_{k = 1}^j m_k$. Note that the number of vertices of $G_j$ is $2^jM^{(j)}$ and the number of edges is $j2^jM^{(j)}$, which follows immediately by induction. We are now in a position to construct a quadrilateral embedding of $\left(\text{\repcube{i}{2r}} \right) \Box \left( \text{\repcycle{j}} \right)$ and, consequently, determine the genus of this graph.
\begin{theorem}\label{thm:mainResult}
    Let $m_{\alpha} \geq 2$, for all $\alpha \in \{1, \ldots, j\}$. Then, the genus of the graph $\left(\text{\repcube{i}{2r}} \right) \Box \left( \text{\repcycle{j}} \right)$ is given by $$g\left(\left(\text{\repcube{i}{2r}} \right) \Box \left( \text{\repcycle{j}} \right)\right) = 1 + M^{(j)}2^{2i + j - 2}r^i(j+ir-2).$$
\end{theorem}
\begin{proof}
    Recall that \repcube{i}{2r} is denoted by $Q_i^{(2r)}$ and \repcycle{j} is denoted by $G_j$. Let $G = Q_{i}^{(2r)}\Box \ G_j$. The proof proceeds by induction on $j$, where the property $P(j)$ is as follows:
    There is a quadrilateral embedding of $G$, with the number of faces $f_{j} = r^i2^{2i+j-1}M^{(j)}(j+ir)$, including two disjoint sets such that each one has $r^i2^{2i+j-2}M^{(j)}$ mutually vertex-disjoint quadrilateral faces and contains all the $r^i2^{2i + j}M^{(j)}$ vertices of $G$.

    For $j = 1$, the quadrilateral embedding of $Q_{i}^{(2r)} \Box \ C_{2m_1}$ is given by Lemma \ref{thm:genus2jCubeC2s}. The number of faces $f_1$ is given by $f_1=\frac{1}{2}\left|E(Q_{i}^{(2r)} \Box \ C_{2m_1})\right|$, and thus
    \begin{align*}
        f_1 &= \frac{(2m_1)(2^{2i}r^i) + (2m_1)(i2^{2i}r^{i+1})}{2}\\
        &= m_12^{2i}r^i + im_12^{2i}r^{i+1}\\
        &= 2^{2i}m_1r^i(1 + ir).
    \end{align*}
    which conforms to the number of expected faces as stated by $P(1)$. The first disjoint set of faces may be chosen by selecting opposite faces from alternate links, and the second set of faces may be formed by choosing the remaining faces from the same handles in each chosen link. Therefore, each set would contain $2m_1(2^{2i - 2}r^i) = 2^{2i - 1}m_1r^i$ quadrilateral faces. This follows since half of the $2m_1$ links are chosen, where each link has $2^{2i - 2}r^i$ handles and only two faces are chosen from each handle. The faces in each set are mutually vertex-disjoint and each set contains all the vertices of $Q_{i}^{(2r)} \Box \ C_{2m_1}$. Therefore, $P(1)$ holds.

    Assuming $P(j)$ is true, we will show that $P(j+1)$ also holds. The graph we are considering is $G' = Q_{i}^{(2r)} \Box \ G_{j+1}$. First, minimally embed $2m_{j+1}$ copies of $G$ as described by $P(j)$. Similarly to the previous result, we partition these copies, corresponding to the vertex-set partition of $C_{2m_{j+1}}$, in two equal parts such that the embeddings of one partition are mirror images of the embeddings in the other partition. From each copy of $G$ from one partition, two links are to be made, both to copies of the other partition. The number of vertices of $G$ is $r^i 2^{2i+j} M^{(j)}$, and thus we need to pass that many edges over a link joining two copies of $G$ with different orientations. Therefore, the number of handles required for each link is $\frac{1}{4}\left(r^i 2^{2i+j} M^{(j)}\right)$ since each handle carries four edges. Using $P(j)$, we have two disjoint sets each of $r^i 2^{2i+j-2} M^{(j)}$ mutually vertex-disjoint quadrilateral faces and thus we have a sufficient number of faces such that we can construct these links using $r^i2^{2i+j-2}M^{(j)}$ handles, where each handle carries four edges so that each new face formed is a quadrilateral. In this manner, the graph $G'$ is embedded after constructing the $2m_{j+1}$ links.

    The number of faces $f_{j+1}$ in $G'$ is $f_{j+1} = 2m_{j+1}f_{j} + \tilde{f}$, where $\tilde{f}$ denotes the number of new faces introduced by the handles. Note that $\tilde{f} = 4m_{j+1}(r^i2^{2i+j-2}$ $M^{(j)})$ since we have $2m_{j+1}$ links, with each link requiring $r^i2^{2i+j-2}M^{(j)}$ handles and, by Lemma \ref{additionalQuadFaces}, the difference in the number of quadrilateral faces is two per handle. Therefore,
    \begin{equation*}
        \begin{split}
            f_{j+1} &= 2m_{j+1}(r^i2^{2i+j-1}M^{(j)}(j+ir)) + 4m_{j+1}(r^i2^{2i+j-2}M^{(j)})\\
            &= r^{i}2^{2i+j}M^{(j+1)}(j+ir) + r^{i}2^{2i+j}M^{(j+1)}\\
            &= r^{i}2^{2i+j}M^{(j+1)}(j+ir+1).
        \end{split}
    \end{equation*}
    This satisfies the expected number of faces as stated by $P(j+1)$. Next, we have to construct the two disjoint sets of faces. The first set of faces may be chosen by selecting opposite faces from alternate links, and the second set of faces is formed by choosing the remaining faces from the same handles in each chosen link. Therefore, each disjoint set would contain $2m_{j+1}(r^i2^{2i+j-2}M^{(j)}) = 2^{2i + j - 1}r^iM^{(j+1)}$ quadrilateral faces. This follows since half of the $2s$ links are chosen, where each link has $r^i2^{2i+j-2}M^{(j)}$ handles and only two faces are chosen from each handle. The faces in each set are mutually vertex-disjoint and each set contains all the vertices of $G'$. Therefore, $P(j+1)$ follows from $P(j)$ and since $G$ is bipartite by Lemma \ref{thm:prodBip}, we calculate the genus using Lemma \ref{thm:genus} to get
    \begin{equation*}
        \begin{split}
            g(Q_{i}^{(2r)}\Box \ G_j) &= 1 + \frac{r^i2^{2i+j}M^{(j)}(j + ir)}{4} - \frac{2^{2i+j}r^iM^{(j)}}{2}\\
            &= 1 + M^{(j)}2^{2i + j - 2}r^i(j+ir-2).
        \end{split}
    \end{equation*}
\qed
\end{proof}

As a direct corollary of the above result, we establish the genus of $K_{2r,2r} \Box G_j$, addressing the case when $r > j$ which is excluded from Pisanski's result.

\begin{corollary}
    The genus of $\left( \text{\repcycle{j}} \right) \Box \ K_{2r,2r}$ is given by $$ g\left(\left( \text{\repcycle{j}} \right) \Box \ K_{2r,2r}\right) = 1 + r2^{j}M^{(j)}(j + r - 2).$$
\end{corollary}

\section{The Cartesian product of the $2r$-cube with paths}
Using Lemma \ref{thm:genus2jCubeC2s}, we can establish the genus of $\left(\text{\repcube{i}{2r}}\right) \Box \ P_{2s}$ as follows.
\begin{lemma}\label{cor:genusRepK2r2rP2s}
    The genus of $\left(\text{\repcube{i}{2r}}\right) \Box \ P_{2s}$ is given by $$g\left(\left(\text{\repcube{i}{2r}}\right) \Box \ P_{2s}\right) = 1 + 2^{2i - 2}r^i(2s(ir - 1) - 1).$$
\end{lemma}
\begin{proof}
    Recall that \repcube{i}{2r} is denoted by $Q_i^{(2r)}$ and let $G=Q_{i}^{(2r)}\Box \ P_{2s}$. Note that $G$ is bipartite by Lemma \ref{thm:prodBip}. We start by embedding $G' = Q_{i}^{(2r)}\Box \ C_{2s}$ as given in Lemma \ref{thm:genus2jCubeC2s}. To obtain $G$, we only need to remove $4^ir^i$ edges from $G'$ between a copy of $Q_{i}^{(2r)}$ found in one partition and another copy from the other partition. The only way these two copies are connected is through $2^{2i-2}r^i$ handles, therefore, we remove these $2^{2i-2}r^i$ handles and the edges on them. The faces we have reintroduced are quadrilateral, since the embedding of $G'$ is quadrilateral, producing a quadrilateral embedding for $G$. The genus of $G$ is given by,
    \begin{align*}
        g(Q_{i}^{(2r)}\Box \ P_{2s}) &= 1 + 2^{2i - 1}sr^i(ir - 1) - 2^{2i - 2}r^i\\
        &= 1 + 2^{2i-2}r^i(2s(ir-1) - 1).
    \end{align*}
\qed
\end{proof}

Similarly to the repeated product of even cycles, we denote by $M^{(j)}$ the product $\prod_{\alpha=1}^{j} m_{\alpha}$ and by $m^{(j)}$ the sum $\sum_{\alpha=1}^{j} \frac{1}{m_{\alpha}}$.
Using Lemma \ref{cor:genusRepK2r2rP2s}, we establish the genus of $\left(\text{\repcube{i}{2r}}\right)\Box \left(\text{\reppath{j}}\right)$.
\begin{theorem} \label{thm:mainResult2}
    The genus of $\left(\text{\repcube{i}{2r}}\right)\Box \left(\text{\reppath{j}}\right)$ is given by $$g\left(\left(\text{\repcube{i}{2r}}\right)\Box \left(\text{\reppath{j}}\right)\right) = 1 + 2^{2i + j - 3}r^iM^{(j)}(2ir + 2j - m^{(j)} - 4).$$
\end{theorem}
\begin{proof}
    Recall that \repcube{i}{2r} is denoted by  $Q_i^{(2r)}$ and that \reppath{j} is denoted by  $H_j$. Let $G=Q_{i}^{(2r)}\Box \ H_j$. The proof proceeds by induction on $j$, where the property $P(j)$ is as follows:
    There is a quadrilateral embedding of $G$, with $f_{j} = 2^{2i+j-2}r^iM^{(j)}(2j+2ir - m^{(j)})$, including two disjoint sets of $r^i2^{2i+j-2}M^{(j)}$ mutually vertex-disjoint quadrilateral faces each, such that each set contains all the $r^i2^{2i + j}M^{(j)}$ vertices of $G$.

    For $j = 1$, the quadrilateral embedding of $Q_{i}^{(2r)}\Box \ P_{2m_1}$ is given by Lemma \ref{cor:genusRepK2r2rP2s}.
    The number of faces $f_1$ is given by $f_1=\frac{1}{2}\left|E(Q_{i}^{(2r)} \Box \ P_{2m_1})\right|$, and thus
     \begin{align*}
        f_1 &= \frac{(2m_1-1)(2^{2i}r^i) + (2m_1)(i2^{2i}r^{i+1})}{2}\\
        &= 2^{2i-1} r^i (2m_1 - 1 + 2m_1 i r)\\
        &= 2^{2i}m_1r^i\left(1 + ir-\tfrac{1}{2m_1}\right).
    \end{align*}
    which conforms to the number of expected faces as stated by $P(1)$. The first disjoint set of faces may be chosen by selecting opposite faces from alternate links, and the second set of faces may be formed by choosing the remaining faces from the same handles in each chosen link. Therefore, each disjoint set would contain $2m_1(2^{2i - 2}r^i) = 2^{2i - 1}r^im_1$ quadrilateral faces. This follows since from $2m_1 - 1$ links, $\lceil \frac{2m_1 - 1}{2} \rceil$ are chosen, where each link has $2^{2i - 2}r^i$ handles and only two faces are chosen from each handle. The faces in each set are mutually vertex-disjoint and each set contains all the vertices of $Q_{i}^{(2r)} \Box \ P_{2m_1}$, therefore, $P(1)$ holds.

    Assuming $P(j)$ is true, we will show that $P(j+1)$ also holds. The graph we are considering is $G' = Q_{i}^{(2r)} \Box \ H_{j+1}$. First, minimally embed $2m_{j+1}$ copies of $G$ as described by $P(j)$. We partition these copies in two equal parts,
    corresponding to the vertex-set partition of $P_{2m_{j+1}}$ (refer to Figure \ref{fig:thm2}) such that the resulting embeddings of one partition are mirror images of the embeddings in the other partition. From each copy of $G$ of the first partition, except for the end copy, two links are to be made, both to copies of the second partition. From the last copy of $G$ of the first partition, one link is made is to a copy in the second partition. We can construct these links using $r^i2^{2i+j-2}M^{(j)}$ handles, where each handle carries four edges so that each new face formed is a quadrilateral. In this way, the graph $G'$ is embedded after constructing the $2m_{j+1}-1$ links.

    \begin{figure}[h!]
        \centering
        \includegraphics[width=\textwidth]{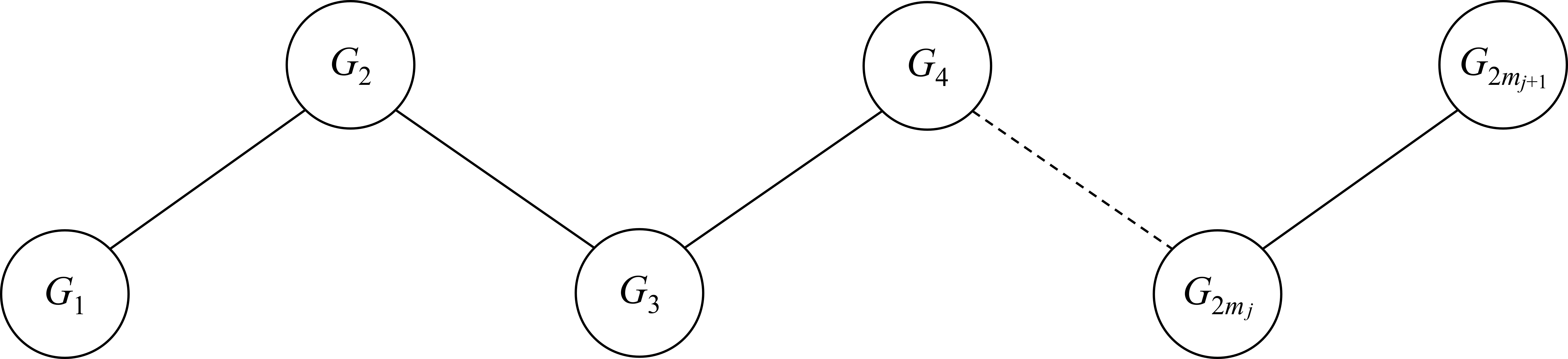}
        \caption{Connecting the links between the two partitions of copies of $G$.}
        \label{fig:thm2}
    \end{figure}

    The number of faces in $G'$ is $f_{j+1} = 2m_{j+1}f_{j} + \tilde{f}$, where $\tilde{f}$ denotes the number of new faces introduced by the handles. Note that \linebreak $\tilde{f} = 2 (2m_{j+1}-1) (r^i2^{2i+j-2} M^{(j)})$ since we have $2m_{j+1} - 1$ links, with each link requiring  $r^i 2^{2i+j-2}M^{(j)}$ handles and, by Lemma \ref{additionalQuadFaces}, the difference in the number of faces is two per handle. Therefore,
    \begin{equation*}
        \begin{split}
            f_{j+1} &= 2m_{j+1}(2^{2i+j-2}r^iM^{(j)}(2j+2ir - m^{(j)})) + (2m_{j+1}-1)(r^i2^{2i+j-2} M^{(j)})\\
            &= r^{i}2^{2i+j-1}M^{(j+1)}\left(2ir + 2j - m^{(j)} + 2 - \frac{1}{m_{j+1}}\right)\\
            &= 2^{2i+j-1}r^{i}M^{(j+1)}(2ir + 2(j+1) - m^{(j+1)}).
        \end{split}
    \end{equation*}
    This satisfies the expected number of faces as stated by $P(j+1)$. Next, we have to construct the two disjoint sets of faces. The first disjoint set of faces may be chosen by selecting opposite faces from alternate links. The second set of faces may be formed by choosing the remaining faces from the same handles in each chosen link. Therefore, the disjoint set would contain $2\left\lceil\frac{2m_{j+1} - 1}{2} \right\rceil(r^i2^{2i+j-2}M^{(j)}) = m_{j+1}2^{2i + j - 1}r^iM^{(j)} = 2^{2i + j - 1}r^iM^{(j+1)}$ quadrilateral faces. This follows since half of the $2m_{j+1} - 1$ links are chosen, where each link has $r^i2^{2i+j-2}M^{(j)}$ handles and only two faces are chosen from each handle. The faces in each set are mutually vertex-disjoint and each set contains all the vertices of $G'$. Therefore, $P(j+1)$ follows from $P(j)$ and since $G$ is bipartite by Lemma \ref{thm:prodBip}, we calculate the genus using Lemma \ref{thm:genus} to get
    \begin{equation*}
        \begin{split}
            g(Q_{i}^{(2r)}\Box \ H_j) &= 1 + \frac{r^i2^{2i+j-1}M^{(j)}(2j + 2ir - m^{(j)})}{4} - \frac{2^{2i+j}r^iM^{(j)}}{2}\\
            &= 1 + r^i2^{2i+j-3}M^{(j)}(2j + 2ir - m^{(j)}) - 2^{2i+j-1}r^iM^{(j)}.\\
            &= 1 + 2^{2i+j-3}r^iM^{(j)}(2ir + 2j - m^{(j)} - 4).
        \end{split}
    \end{equation*}
\qed
\end{proof}

We remark that Pisanski's result also excludes Theorem \ref{thm:mainResult2} since the path is not a regular graph.


%
%



\end{document}